\documentclass[10pt, reqno]{amsart}

\usepackage{amsmath,amssymb,amsthm,amsfonts,verbatim}
\usepackage{graphicx}
\usepackage{pinlabel}
\usepackage{hyperref}
\usepackage{color}
\usepackage{enumerate}
\usepackage{cite}
\usepackage{subcaption}
\usepackage{tikz}
\usetikzlibrary{arrows.meta}
\usepackage{xfrac}
\usepackage{physics}

\usepackage[top=1.2in,bottom=1.8in,left=1.2in,right=1.2in]{geometry}

\theoremstyle{plain}
\newtheorem{theorem}{Theorem}[section]

\newtheorem{question}[theorem]{Question}

\newtheorem{proposition}[theorem]{Proposition}
\newtheorem{lemma}[theorem]{Lemma}

\newtheorem{corollary}[theorem]{Corollary}

\theoremstyle{definition}
\newtheorem{definition}[theorem]{Definition}
\newtheorem{example}[theorem]{Example}
\newtheorem{construction}[theorem]{Construction}

\newtheorem{remark}[theorem]{Remark}

\renewcommand{\epsilon}{\varepsilon}

\renewcommand{\le}{\leqslant}
\newcommand{\inv}{^{-1}}
\newcommand{\Q}{\mathcal{Q}}

\newcommand{\R}{\mathbb{R}}

\newcommand{\C}{\mathbb{C}}
\newcommand{\N}{\mathbb{N}}

\renewcommand{\H}{\mathcal{H}}

\title{Oblivious Points on Translation Surfaces}

\author{Ian Adelstein\textsuperscript{*}, Krish Desai, Anthony Ji, Grace Zdeblick}

\thanks{\noindent\textsuperscript{*}corresponding author}

\address{Department of Mathematics, Yale University, 10 Hillhouse Ave, New Haven, CT 06511}

\begin{document}

\begin{abstract}
An oblivious point on a translation surface is a point with no closed geodesic passing through it. Nguyen, Pan, and Su \cite{NPS} showed that there are at most finitely many oblivious points on any given translation surface and constructed a family of surfaces with exactly one oblivious point. We construct new families of translation surfaces with arbitrarily many oblivious points and prove that there is a translation surface in every genus $\ge3$ with an oblivious point.

\keywords{}
\end{abstract}

\maketitle

\section{Introduction}
Translation surfaces are closely related to polygonal billiards and illumination problems. They are surfaces naturally equipped with a flat metric away from a finite set of singularities called cone points.
Due to the flat metric, geodesics on these surfaces are straight lines. Closed geodesics are geodesics which close up on themselves smoothly. Nguyen, Pan, and Su \cite{NPS} investigate the relationship between regular (non-cone) points and closed geodesics. They prove on any translation surface, that the set of regular points not contained in a closed geodesic is finite. Let us call these points \textit{oblivious points}.

\begin{definition}
An \textit{oblivious point} is a regular point on a translation surface through which there are no simple closed geodesics.
\end{definition}

\begin{theorem}[\cite{NPS}, Theorem 1]
The set of oblivious points on a translation surface is finite.
\end{theorem}

Additionally, they give a construction for a family of surfaces with exactly one oblivious point. A natural question to ask is whether it is possible to find translation surfaces with any number of oblivious points. One may also study the genus of surfaces with oblivious points.
In this paper, we demonstrate the following results:
\begin{theorem}\label{n obl}
For every $n \in \mathbb{N}$ there is a translation surface with exactly $n$ oblivious points.
\end{theorem}
\begin{theorem}\label{g geq 3}
There exist translation surfaces in every genus $g \geq3$ admitting an oblivious point.
\end{theorem}
It is known that translation surfaces with genus $g<3$ do not admit oblivious points \cite{NPS, zorich}, so in this sense Theorem~\ref{g geq 3} is optimal.

The paper proceeds as follows: In Section~\ref{background} a standard background on translation surfaces is presented. In Section~\ref{sec:trans} the behaviour of oblivious points under translation covering maps is studied. In Section~\ref{sec:slits} we use the covering technology developed, together with a slit construction and an existence result from \cite{NPS} to prove Theorem~\ref{n obl}. We also show how these techniques can be used to produce a translation surface admitting an oblivious point in every genus $g\geq 4$. In Section~\ref{sec:block} we focus our attention on square-tiled surfaces, and introduce a new construction that allows us to produce translation surfaces admitting an oblivious point in every genus $g \geq 3$, completing the proof of Theorem~\ref{g geq 3}. This construction also yields translation surfaces admitting $n$ oblivious points for every $n \in \mathbb{N}$, thus providing an alternate proof of Theorem~\ref{n obl}. Finally, in Section~\ref{sec:tilings} we prove results about which square-tiled translation surfaces admit oblivious points. 

\section{Background}\label{background}

This section includes much of the background necessary to state and prove our main results. 
 
\subsection{Translation Surfaces}

A \textit{translation surface} is a collection of polygons in the plane with parallel sides identified by complex translation $z\mapsto z + c\in \C,$ up to cut and paste. When one identifies sides by translation, certain singular points of concentrated curvature, called \textit{cone points}, do not have a total angle of $2\pi$ around them. The angles around these cone points, called \textit{cone angles}, are integer multiples of $2\pi$. A \textit{square--tiled surface (STS)} is a translation surface of which every polygon is a square.

 The genus of a surface can be calculated using only the cone angles through the following theorem.

\begin{theorem}[Combinatorial Gauss--Bonnet]
Consider a translation surface with cone points $p_i$ and corresponding cone angles $\alpha_i$. Define the cone angle excess at each point as $k_i=\frac{\alpha_i-2\pi}{2\pi}$. Then
$$\sum_i k_i = 2g-2$$
where $g$ is the genus of the surface.
\end{theorem}

Using the same definition for the $k_i$, surfaces of the same genus may be further partitioned into \textit{strata}.
If $X$ is a translation surface with cone angle excesses $\{k_1,k_2,\cdots,k_n\}$, then the surface is in the stratum denoted $\mathcal{H}(k_1,k_2,\cdots,k_n)$.

\subsection{$GL_2(\R)$ and Veech Group Action}
The natural action of the group $GL_2(\R)$ on a stratum of translation surface is obtained by $GL_2(\R)$ acting on the polygonal representation of the surface in $\R^2$.

One might seek to study the induced action of this group on points and geodesics on a translation surface. However, such an action is not quite induced on points and geodesics on the surface, because the cut and paste action described above can identify points on a translation surface with a non-trivial translation automorphism group. Rather, this action is on equivalence classes of points and geodesics modulo this group, rather than directly on points. To resolve this ambiguity, one typically passes to a finite cover of the stratum of translation surfaces, and by marking sufficiently many points one can break the equivalence and distinguish the surfaces. Here the action is indeed on points and geodesics. Since the obliviousness of a point commutes with translation automorphisms, in practice, one can pass freely between the the stratum and its finite cover, though precision about this process is indeed required in formal proof.

\begin{proposition}
Consider translation surface $X$ in stratum $ \H$ with an oblivious point at $x$, and translation automorphism group $G$. Let $\phi: \overline\H \to \H$ be a covering map that breaks the ambiguity created by $G$ by adding marked points. For every $g\in GL_2(\R)$ if $X_g$ is an element of then set $\phi(g\cdot \phi\inv (X)),$ and $x_g$ is an element of the set $\phi(g\cdot \phi\inv(x))$ the $x_g$ is oblivious in $X_g.$
\end{proposition}
\begin{proof}
Let $X$ be a translation surface and $x\in X$ an oblivious point. By way of contradiction, assume there exists $g\in GL_2(\R)$ such that $y \in \phi(g\cdot \phi\inv(x))$ is contained in a closed geodesic $\gamma$. Since $GL_2(\R)$ is a group, the element $g$ is invertible. This implies that there is a closed geodesic in $\phi(g^{-1}\cdot\phi\inv( \gamma))$ through $x = \phi(g^{-1}\cdot \phi\inv(y))$. But $x$ is an oblivious point by assumption, which leads to a contradiction.
\end{proof}

One may summarize the above proposition through the slogan, ``the $GL_2(\R)$ image of an oblivious point is oblivious."
\begin{definition}
Given a translation surface $X$, its \textit{Veech group}, $SL(X)$ is the stabiliser subgroup of $X$  under $SL_2(\R).$ It comprises those elements of $SL_2(\R)$ whose action sends $X$ back to itself up to cut--and--paste. 
\end{definition}

\begin{corollary}
As $SL(X)$ is a subgroup of $GL_2(\R)$, the orbit of an oblivious point under the Veech group action consists of oblivious points. Let $\operatorname{Trans}(X)$ be the group of translation automorphisms of $X.$ Given (a $\operatorname{Trans}(X)$-orbit of) an oblivious point on a translation surface $X$ with non-trivial Veech group, additional oblivious points may be obtained by identifying the points in the orbit  with points on $X$. 

\end{corollary}

\subsection{Covers}

A \textit{cover} of a surface $X$ is a pair $(Y, \rho)$ where $Y$ is a closed surface and $\rho: Y\to X$ is a surjective local homeomorphism. The \textit{degree} of a cover $(Y, \rho)$ of a surface $X$ is the cardinality of the fiber of $\rho$ over any regular point $x\in X$, 
$$\deg \rho = |\{\rho^{-1}\{x\}\}|.$$
A \textit{branched cover} is a cover which ramifies over a finite set; that is, $(Y, \rho)$ is a branched cover of a closed surface $X$ if there exist finite sets of points, \textit{branch points} $B\subseteq Y$  and \textit{branching values} $V\subseteq X$, such that $\rho: Y\setminus B \to X\setminus V$ is a covering map. A \textit{translation cover} is a cover, branched or otherwise, which preserves translation structure. In a polygonal representation, a translation cover can be represented as multiple copies of a surface with edge identification preserved between copies. In this paper, we explore primarily branched translation covers.

A \textit{fully ramified} cover is a cover for which each of the branch points has cone angle that exceeds $2\pi$. One may note that translation covers are local isometries.

\subsection{Semi-Translation Surfaces}
Similar to translation surfaces, \textit{semi-translation surfaces}, also known as half-translation surfaces, are defined as collections of polygons in the plane with sides identified by translations and $180\deg$-rotations ($z\mapsto \pm z + c$) up to cut--and--paste. All cone points on a semi-translation surface have cone angles which are integer multiples of $\pi$.

The strata for semi-translation surfaces are defined similarly to those for translation surfaces. For a semi-translation surface $X$ with cone angles $\alpha_i$, define $l_i=\frac{\alpha_i-2\pi}{\pi}$ and say that $X$ is in $\mathcal{Q}(l_1, l_2, \ldots, l_n)$. 

Any given semi-translation surface can be related to a translation surfact through the following construction. The \textit{canonical double cover} of a semi-translation surface is defined according to Zorich\cite{zorich}: a semi-translation surface in the stratum $\mathcal{Q}(l_1, \ldots, l_n)$ is mapped to a surface in the stratum $\mathcal{H}(k_1,\ldots,k_m)$ where the $\{k_i\}$ are obtained from the $\{l_j\}$ via the following rules: 
\begin{enumerate}
\item To each even $l_j>0$, associate the pair $\{\frac{1}{2}l_j, \frac{1}{2}l_j\}$ in the $\{k_i\}$. 

\item To each odd $l_j>0$, associate $\{l_j+1\}$ in the $\{k_i\}$.

\item No $k_i$ is associated for $l_j=-1$.
\end{enumerate}
In words, cone angles that are even multiples of $\pi$ are duplicated in number and cone angles that are odd multiples of $\pi$ are doubled in angle. As such, the canonical double cover of a semi-translation surface is a translation surface.

As the first existence result for oblivious points Nguyen, Pan, and Su \cite{NPS} showed that the canonical double cover can be used to construct oblivious points via the following theorem. 

\begin{theorem}[\cite{NPS}, Proposition 3.1]\label{NPS3.1}
Let $X$ be a semi-translation surface with exactly one cone point $y$ of cone angle $\pi$ (other cone points with different cone angles are allowed). The canonical double cover of $X$ has an oblivious point at the pre-image of $y$.
\end{theorem}
\begin{lemma}
[\cite{NPS}, Remark 3.2i]\label{NPS3.2i}
If $X$ has the additional property that every regular point is contained in a simple closed geodesic, then the preimage of $y$ is the unique oblivious point on the canonical double cover of $X$.
\end{lemma}

A visual representation of the canonical double cover of a semi-translation surface can be constructed with the following algorithm. Note that the goal is to eliminate the reflections, i.e.~to take sides that had been identified via translation and reflection, and reassign the identification in a way that eliminates the reflection. 

\begin{enumerate}
    \item Mark sides identified by translation with numbers, and sides identified by translation and reflection with opposite-facing arrows.
    \item Take two copies of the surface. Sides identified by translation will be associated as originally. Relabel these sides on one copy of the surface to avoid ambiguity. Leave the arrows alone.
    \item Rotate one of the copies by $180^\text{o}$.
    \item Leave the association of the numbered sides. Associate sides with arrows with their same orientation counterpart; this eliminates any reflection identification. 
\end{enumerate}

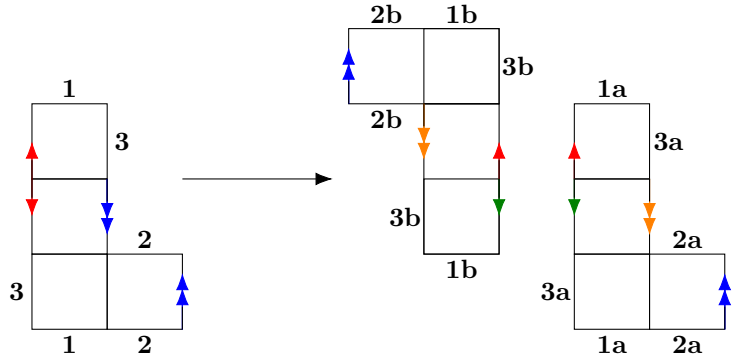
\begin{figure}[h!]
\begin{center}
    \begin{tikzpicture}
    \definecolor{ao}{rgb}{0.0, 0.5, 0.0}
    \draw (0, 0) rectangle (1,1);
    \draw (1, 0) rectangle (2, 1);
    \draw (0, 1) rectangle (1, 2);
    \draw (0, 2) rectangle (1, 3);
    \draw[draw=red,-{Latex[length=2.3mm]}] (0, 2) -- (0, 2.5);
    \draw (0,2) -- (0,2.25);
    \draw[draw=red,-{Latex[length=2.3mm]}] (0, 2) -- (0, 1.5);
    \draw (0,2) -- (0,1.75);
    \draw[draw=blue,-{Latex[length=2.3mm]}] (1, 2) -- (1, 1.25);
    \draw[draw=blue,-{Latex[length=2.3mm]}] (1, 2) -- (1, 1.45);
    \draw (1,2) -- (1,1.7);
    \draw[draw=blue,-{Latex[length=2.3mm]}] (2, 0) -- (2, 0.55);
    \draw[draw=blue,-{Latex[length=2.3mm]}] (2, 0) -- (2, .75);
    \draw (2,0) -- (2,0.3);
    \draw[-{Latex[length=2.3mm]}] (2, 2) -- (4, 2);
    \node[] at (0.5,-0.2) {\textbf{1}};
    \node[] at (0.5,3.2) {\textbf{1}};
    \node[] at (1.5,-0.2) {\textbf{2}};
    \node[] at (1.5,1.2) {\textbf{2}};
    \node[] at (-0.2,0.5) {\textbf{3}};
    \node[] at (1.2,2.5) {\textbf{3}};
    \end{tikzpicture}
   \begin{tikzpicture}
    \definecolor{ao}{rgb}{0.0, 0.5, 0.0}
    \draw (1, 0) rectangle (2, 1);
    \draw (0, 1) rectangle (-1, 2);
    \draw (2, 0) rectangle (3, 1);
    \draw (1, 1) rectangle (2, 2);
    \draw (0, 1) rectangle (-1, 3);
    \draw (1, 2) rectangle (2, 3);
    \draw (0, 3) rectangle (-1, 4);
    \draw (0, 3) rectangle (-2, 4);
    
    \draw[draw=red,-{Latex[length=2.3mm]}] (0, 2) -- (0, 2.5);
    \draw (0,2) -- (0,2.25);
    \draw[draw=ao,-{Latex[length=2.3mm]}] (0, 2) -- (0, 1.5);
    \draw (0,2) -- (0,1.75);
    \draw[draw=red,-{Latex[length=2.3mm]}] (1, 2) -- (1, 2.5);
    \draw (1,2) -- (1,2.25);
    \draw[draw=ao,-{Latex[length=2.3mm]}] (1, 2) -- (1, 1.5);
    \draw (1,2) -- (1,1.75);

    \draw[draw=orange,-{Latex[length=2.3mm]}] (2, 2) -- (2, 1.25);
    \draw[draw=orange,-{Latex[length=2.3mm]}] (2, 2) -- (2, 1.45);
    \draw (2,2) -- (2,1.7);
    \draw[draw=orange,-{Latex[length=2.3mm]}] (-1, 3) -- (-1, 2.25);
    \draw[draw=orange,-{Latex[length=2.3mm]}] (-1, 3) -- (-1, 2.45);
    \draw (-1,3) -- (-1,2.7);
    \draw[draw=blue,-{Latex[length=2.3mm]}] (3, 0) -- (3, 0.75);
    \draw[draw=blue,-{Latex[length=2.3mm]}] (3, 0) -- (3, .55);
    \draw (3,0) -- (3,0.3);
    \draw[draw=blue,-{Latex[length=2.3mm]}] (-2, 3) -- (-2, 3.75);
    \draw[draw=blue,-{Latex[length=2.3mm]}] (-2, 3) -- (-2, 3.55);
    \draw (-2,3) -- (-2,3.3);

    \node[] at (-0.5,0.8) {\textbf{1b}};
    \node[] at (-0.5,4.2) {\textbf{1b}};
    \node[] at (-1.5,2.8) {\textbf{2b}};
    \node[] at (-1.5,4.2) {\textbf{2b}};
    \node[] at (-1.25,1.5) {\textbf{3b}};
    \node[] at (0.25,3.5) {\textbf{3b}};
    \node[] at (1.5,-0.2) {\textbf{1a}};
    \node[] at (1.5,3.2) {\textbf{1a}};
    \node[] at (2.5,-0.2) {\textbf{2a}};
    \node[] at (2.5,1.2) {\textbf{2a}};
    \node[] at (0.75,0.5) {\textbf{3a}};
    \node[] at (2.25,2.5) {\textbf{3a}};
    \end{tikzpicture}
    \caption{An example of the canonical double cover algorithm}
\end{center}
\end{figure}

This construction duplicates cone angles which are even multiples of $\pi$ and doubles cone angles which are odd multiples of $\pi$, following the involution map described above. Thus the canonical double cover of a semi-translation surface is indeed a translations surface.

\section{Translation Covers and Geodesics}\label{sec:trans}

In this section we give results on the behavior of closed geodesics under translation covers. We will examine the behavior of preimages of points under translation covers and how it affects the obliviousness of a point. 
\begin{lemma}\label{Cover Converse}
The preimage of a point in a non-singular closed geodesic under a fully ramified translation cover is contained in a non-singular closed geodesic.
\end{lemma} 
\begin{proof}
    Let $(X, P)$ be a translation surface with marked points $P = \{p_1,p_2,\cdots,p_k\}$ and let $Y$ be a branched translation cover of $X$ with branched covering map $\rho$, branching points $B \subseteq Y$, and branching values $V\subseteq X$ with $P \subseteq V$. Let $\gamma$ be a closed non-singular geodesic in $X\setminus V$. Let $x$ be a point in $\gamma$ and $y \in \rho^{-1}(x)$. Denote $\tau$ to be the connected component of $\rho^{-1}(\gamma)$ containing $y$. Note that $\tau$ is a geodesic since $\rho$ is a local isometry and that $\rho(\tau) = \gamma$. Then $\tau$ is a closed non-singular geodesic in $Y\setminus B$. 
\end{proof}
The above lemma equivalently states that the preimage of a non-oblivious point under a translation cover is not oblivious. We also demonstrate the converse of this in the following lemma.

\begin{lemma}\label{obl}
The preimage of an oblivious point under a fully ramified translation cover is oblivious.
\end{lemma}
\begin{proof}
It suffices to prove the equivalent statement that the image of a non-singular closed geodesic under a translation cover is a non-singular closed geodesic.
Let $(X, P)$ be a translation surface with marked points $P = \{p_1,p_2,\cdots,p_k\}$. Let $Y$ be a branched translation cover of $X$ with $\rho$ the branched covering map with branching points $B \subseteq Y$ and branching values $V\subseteq X$ with $P \subseteq V$.
We have $B = \rho^{-1}(V)$ by definition.
Suppose $\gamma \subseteq Y \setminus B$ is a simple, closed, non-singular geodesic (since all points in $B$ have cone angle $4\pi$ or greater, we need not consider closed geodesics containing points in $B$). Then since $\rho$ is continuous, $\rho(\gamma)$ is connected as $\gamma$ is connected. Hence $\rho(\gamma)$ is closed. Moreover $\rho$ is a local isometry, thus $\rho(\gamma)$ is a geodesic. Now suppose $\rho( \gamma)$ is singular. Then by assumption $\rho(\gamma)$ passes through a point in $V$. Since $\gamma \cap B = \emptyset$ this is impossible. Thus $\rho (\gamma)$ is non-singular.
\end{proof}

We use these lemmas to better understand the relationship between oblivious points and regularly tiled polygonal surfaces.

\begin{lemma}\label{polygon vertex}
Let $X$ be a translation surface tiled by a regular polygon. Then any oblivious point of $X$ must be a vertex point.
\end{lemma}

While this lemma could be shown by considering periodic points on the regular $n$--gon, here we provide an elementary proof to show that all non-vertex points are contained in some cylinder.

\begin{proof}
Suppose $X$ is tiled by a regular $n$--gon. We consider two cases.\\
\indent
Case 1: If $n$ is even, $X$ is a branched translation cover of the regular $n$--gon, whose opposite sides are identified, with the vertices of the $n$--gon as the branching values of the translation cover. Any non-vertex point is the preimage of a non-vertex point in the base $n$--gon. We show that just three families of cylinders contain all non-vertex points in the base $n$-gon.

For simplicity, consider the representation of the $n$--gon in the Cartesian plane where the center of the polygon is at the origin and a pair of opposite vertices are on the $y$--axis. Pick one of the edges adjacent to a vertex on the $y$--axis and consider the cylinder perpendicular to this edge together with all parallel cylinders. The blue lines in Figures \ref{regular 4ngon} and \ref{regular 4n+2gon} illustrate the only points not contained in this family of cylinders. Now pick the other pair of edges adjacent to the $y$--axis and consider its associated family of cylinders. The red lines in Figures \ref{regular 4ngon} and \ref{regular 4n+2gon} illustrate the only points not contained in this family of cylinders.

For each $x$ value the blue (alternatively, red) lines contain either one non-vertex point or two vertex points. By the reflection symmetry about the $x$--axis we have that these colored lines intersect only at vertex points and along the $x$--axis (for the case of $n=4$, the intersections are only on vertex points). We need only show that these $x$--axis intersections are contained in some third family of cylinders. 

If $n=4N + 2, N\in \N,$ there is a pair of edges perpendicular to the $x$--axis and the $x$--axis intersections are contained in the associated cylinder. If $n=4N, N\in\N,$ the $n$--gon has two vertices on the $x$--axis, and the $x$--axis intersections are contained in the cylinder associated to one of the edges adjacent to a vertex on the $x$--axis. These cylinders are shown in green in Figures \ref{regular 4ngon} and \ref{regular 4n+2gon}. Thus, every non-vertex point is contained in some closed geodesic.

\indent
Case 2: If $n$ is odd, $X$ is a branched translation cover of the doubled regular $n$--gon. The doubled regular $n$--gon is a presentation of such a surface in $\R^2$ with two regular $n$--gons with one shared (or ``glued'') side, and the parallel edges identified. We show that any non-vertex point in the base doubled regular $n$--gon is contained in just two families of cylinders. 

Take any pair of associated parallel edges and consider the cylinder decomposition parallel to the lines which connect the identified vertices on this pair of edges. The only non-vertex points not contained this family of cylinders are shown by the blue lines in Figure \ref{regular 2n+1gon}. Now consider an edge which shares a vertex with the first pair of edges and is not the ``glued'' side of the doubled $n$--gon, as well as its associated parallel edge. Again, consider the cylinder decomposition parallel to the lines which connect the identified vertices on this pair of edges. By the choice of using an adjacent edge, the boundaries of this family of cylinders, shown in red in Figure \ref{regular 2n+1gon}, intersect with the boundaries from the first family of cylinders only at vertex points. Thus, every non-vertex point is contained in a cylinder from one of these two sets of cylinders.

By Lemma $\ref{Cover Converse}$ we conclude that only vertex points may be oblivious.
\end{proof}

\begin{figure}[!h]
\begin{minipage}{.5\linewidth}
\centering
\begin{tikzpicture}
     \draw (2,0) -- (1.7321, 1)  -- (1,1.7321) -- (0,2) -- (-1,1.7321) -- (-1.7321, 1) -- (-2, 0) -- (-1.7321, -1) -- (-1, -1.7321) -- (0,-2) -- (1, -1.7321) -- (1.7321, -1) -- (2,0);
     
     \draw[color=green, fill=green!40, fill opacity = 0.5] (-1.7321, 1) -- (2,0) -- (1.7321, -1) -- (-2,0) -- (-1.7321, 1);
     
     \draw[color=blue] (0, 2) -- (1, -1.7321);
     \draw[color=blue] (-1, 1.7321) -- (0, -2);
     \draw[color=blue] (-1.7321, 1) -- (-1, -1.7321);
     \draw[color=blue] (-2, 0) -- (-1.7321, -1);
     \draw[color=blue] (1, 1.7321) -- (1.7321, -1);
     \draw[color=blue] (1.7321, 1) -- (2,0);
     
     \draw[color=red] (0, -2) -- (1, 1.7321);
     \draw[color=red] (-1, -1.7321) -- (0, 2);
     \draw[color=red] (-1.7321, -1) -- (-1, 1.7321);
     \draw[color=red] (-2, 0) -- (-1.7321, 1);
     \draw[color=red] (1, -1.7321) -- (1.7321, 1);
     \draw[color=red] (1.7321, -1) -- (2,0);
     
     \filldraw[black] (-0.53,0) circle (2pt);
     \filldraw[black] (0.53,0) circle (2pt);
     \filldraw[black] (-1.48,0) circle (2pt);
     \filldraw[black] (1.48,0) circle (2pt);
    \end{tikzpicture}
    \caption{Regular $4n$-gon with cylinder boundaries}
    \label{regular 4ngon}
\end{minipage}%
\begin{minipage}{.5\linewidth}
\centering
\begin{tikzpicture}
    
    \draw (0,2) -- (1.176, 1.618) -- (1.902, 0.618) -- (1.902, -0.618) -- (1.176, -1.618) -- (0,-2) -- (-1.176, -1.618) -- (-1.902, -0.618) -- (-1.902, 0.618) -- (-1.176, 1.618)-- (0,2);
     
     \draw[color=green, fill=green!40, fill opacity = 0.5] (1.902, -0.618) -- (1.902,-0.618) -- (-1.902, -0.618) -- (-1.902, 0.618) -- (1.902, 0.618);
     
     \draw[color=blue] (0, 2) -- (1.176, -1.618);
     \draw[color=blue] (-1.176, 1.618) -- (0, -2);
     \draw[color=blue] (-1.902, 0.618) -- (-1.176, -1.618);
     \draw[color=blue] (1.176, 1.618) -- (1.902, -0.618);
     
     \draw[color=red] (0, -2) -- (1.176, 1.618);
     \draw[color=red] (-1.176, -1.618) -- (0, 2);
     \draw[color=red] (-1.902, -0.618) -- (-1.176, 1.618);
     \draw[color=red] (1.176, -1.618) -- (1.902, 0.618);
     
     \filldraw[black] (-0.65,0) circle (2pt);
     \filldraw[black] (0.65,0) circle (2pt);
     \filldraw[black] (-1.7,0) circle (2pt);
     \filldraw[black] (1.7,0) circle (2pt);
\end{tikzpicture}
    \caption{Regular $4n+2$-gon with cylinder boundaries}
    \label{regular 4n+2gon}
\end{minipage}\par\medskip
\centering
\begin{tikzpicture}
    
    \draw (0,0.684) -- (-0.879,1.732) -- (-2.227,1.97) -- (-3.411, 1.286) -- (-3.879,0) -- (-3.411, -1.286) -- (-2.227, -1.97) -- (-0.879, -1.732) -- (0,-0.684) -- (0,0.684);
    \draw (0,0.684) -- (0.879,1.732) -- (2.227,1.97) -- (3.411, 1.286) -- (3.879,0) -- (3.411, -1.286) -- (2.227, -1.97) -- (0.879, -1.732) -- (0,-0.684) -- (0,0.684);
    
    \draw[-{Latex[length=3mm, fill=blue]}] (-3.879,0) -- (-3.567, 0.857);
    \draw[-{Latex[length=3mm, fill=blue]}] (3.411, -1.286) -- (3.723, -0.429);
    
    \draw[-{Latex[length=3mm, fill=red]}] (-3.411, 1.286) -- (-2.622, 1.742);
    \draw[-{Latex[length=3mm, fill=red]}] (-3.411, 1.286) -- (-2.819, 1.628);
    
    \draw[-{Latex[length=3mm, fill=red]}] (2.227, -1.97) -- (3.016, -1.514);
    \draw[-{Latex[length=3mm, fill=red]}] (2.227, -1.97) -- (2.819, -1.628);
    
    \draw[color=blue] (-3.411, 1.286) -- (3.879,0);
    \draw[color=blue] (-3.879,0) -- (3.411, -1.286);
    \draw[color=blue] (-0.879, -1.732) -- (-3.411, -1.286);
    \draw[color=blue] (0.879,1.732) -- (3.411, 1.286);
    \draw[color=blue]  (-0.879,1.732) -- (-2.227,1.97);
    \draw[color=blue] (2.227, -1.97) -- (0.879, -1.732);
    
    \draw[color=red] (-2.227,1.97) -- (3.411, -1.286);
    \draw[color=red] (-3.411, 1.286) -- (2.227, -1.97);
    \draw[color=red] (-3.879,0) -- (-0.879, -1.732);
    \draw[color=red] (3.879,0) -- (0.879, 1.732);
    \draw[color=red] (-3.411, -1.286) -- (-2.227, -1.97);
    \draw[color=red] (3.411, 1.286) -- (2.227, 1.97);
    
\end{tikzpicture}
    \caption{Doubled regular $2n+1$-gon with chosen edge pairs and cylinder boundaries}
    \label{regular 2n+1gon}
\end{figure}

\begin{proposition}\label{hex}
Any translation surface tiled by a regular polygon other than a triangle, square, or hexagon cannot admit an oblivious point.
\end{proposition}

\begin{proof}
Suppose $X$ is a translation surface tiled by a regular $n$-gon. The angle at each vertex of a regular $n$-gon is $\frac{(n-2)\pi}{n}$. Thus, at any vertex point on $X$, the angle must be a multiple of $\frac{(n-2)\pi}{n}$. However, for $n \neq 3,4,6$, $\frac{(n-2)\pi}{n}$ does not divide $2\pi$. Therefore no vertex point on $X$ is a regular point and by Lemma $\ref{polygon vertex}$ we see that $X$ cannot admit an oblivious point.
\end{proof}
\begin{remark}
Any translation surface tiled by equilateral triangles or regular hexagons is in the ${GL}_2(\mathbb{R})$ orbit of a square tiled surface and thus is also arithmetic. Proposition \ref{hex} therefore implies that any surface tiled by a regular polygon which contains an oblivious point must be arithmetic. When investigating oblivious points on translation surfaces tiled by regular polygons, we therefore need only consider square tiled surfaces.
\end{remark}

\section{Connected Translation Surfaces via Slits}\label{sec:slits}
Nguyen, Pan, and Su \cite[Example 1]{NPS}  construct a translation surface with exactly one oblivious point. One may combine this result with Lemma \ref{obl} to get a surface with $n$ oblivious points as follows. Let $X$ be a translation surface with exactly one oblivious point $x$. By Lemma \ref{obl}, for any degree $n$ translation cover of $X$, the preimages of $x$ are all oblivious. In a degree $n$ cover, there are $n$ preimages to any regular point, so a degree $n$ translation cover has $n$ preimages of $x$. For the trivial $n-$ translation cover of $n$ disconnected copies of $X$, this covering space would have $n$ oblivious points. This produces a disconnected translation surface with $n$ oblivious points, a somewhat trivial proof of our Theorem~\ref{n obl}. 

However, one would prefer to consider connected translation surfaces, and we extend this result to connected translation surfaces. One may build a connected surface with $n$ oblivious points simply by taking unbranched translation $n$--covers of the surface with one oblivious point demonstrated in~\cite{NPS}. However, by computing the Euler characteristic of these $n$--covers, one may observe that one cannot obtain a surface of every genus through this process. In order to provide a fresh construction of connected translation surfaces with $n$ oblivious points for every positive integer $n,$ and connected translation surfaces in every genus $g\in\N,$  we first introduce the notion of a \emph{slit}.

\begin{construction}[Slits]
Given two disconnected translation surfaces,  begin by choosing equal length embedded parallel straight line segments on each surface. Slit the surfaces at these segments, then glue so that the left side of one segment is identified with the right side of the other segment. The result is a connected translation surface with two additional cone points.
\end{construction}

\begin{example}
Given two disconnected tori, as in Figure~\ref{slit tori}, cut each torus along the parallel segments and glue according to the diagram. Geodesics entering the slit on one surface emerge from the associated slit on the other surface with the same trajectory as shown. In this example the two endpoints of the segment (whose copies are identified under the gluing) become a pair of cone points each with angle excess $2\pi$. 
\end{example}

\begin{figure}[h!]
\begin{center}
    \begin{tikzpicture}
    \draw[color=red] (0,0) -- (4,0);
    \draw[color=blue](4, 0) -- (4,4);
    \draw[color=red](4, 4) --(0, 4);
    \draw[color=blue](0, 4)--(0, 0);
  
    \draw[color=green] (5,0) -- (9,0);
    \draw[color=orange](9, 0) -- (9,4);
    \draw[color=green](9, 4) --(5, 4);
    \draw[color=orange](5, 4)--(5, 0);
   
   \draw (2.5, 0.5) -- (1, 2);
   \draw (7.5, 0.5) -- (6, 2);
    \draw[-{Latex[length=3mm]}] (0, 0) -- (1.5, 1.5);
    \draw[-{Latex[length=3mm]}] (6.5, 1.5) -- (7.5, 2.5);
    \draw (1.7, 0.9) node{\textbf{+}};
    \draw (2.3, 1.1) node{$\mathbf{-}$};
    \draw (7.3, 1.1) node{\textbf{+}};
    \draw (6.7, 0.9) node{$\mathbf{-}$};
    \end{tikzpicture}
    \caption{A pair of slit tori and a geodesic}
    \label{slit tori}
\end{center}
\end{figure}
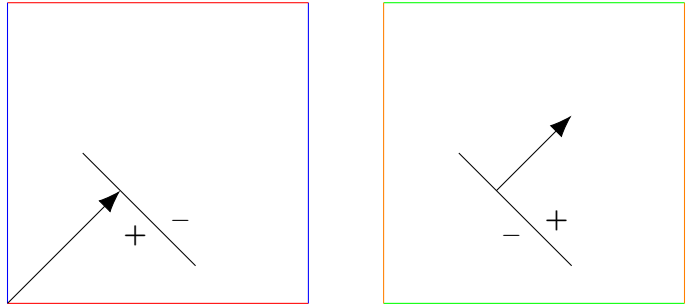

 This slit construction (together with the existence result from \cite[Example 1]{NPS}) is now used to prove the Theorem~\ref{n obl}:

\begin{proof}[Proof of Theorem~\ref{n obl}]
Let $X$ be a translation surface with exactly one oblivious point, $x$. Label $n$ disconnected copies of $X$ as $X_1, X_2, \ldots, X_n$. Choose a closed geodesic $\gamma$ on $X$ that does not contain any cone points or the oblivious point $x$. Let $y$ and $z$ be any two points on $\gamma$ and choose one segment of $\gamma$ with endpoints $y$ and $z$. Call this segment $E$.

Define $y_i$ and $z_i$ on $X_i$ to be the preimages of $y$ and $z$ under the trivial disjoint degree $n$ covering map from $\bigcup\limits_{i=1}^{n} X_i$ to $X$, and similarly define $E_i$ as the preimage of edge $E$. Connect the $X_i$ by ``slitting" them together at the edges $E_i$.

Associate slits such that each edge $E_i$ leads into the edge $E_{i+1 (mod~n)}$ (or any permutation of associating the edges which is an $n$-cycle). When traversing from $E_{i}$ to $E_{i+1}$, a geodesic emerges on $X_{i+1}$ in the equivalent position that it would have continued on $X_i$ were the edges not associated. Under this association the identified $y_i$ and $z_i$ each become cone points of angle $2\pi n$. One has therefore connected the disjoint $X_i$, and this connected space is now a degree $n$ connected translation cover for $X$. By Lemma \ref{obl}, the $n$ preimages of $x$ are oblivious. 
\end{proof}

\begin{theorem}\label{genus1}
For every $g\ge 4$ there exists a surface $Y$ of genus $g$ with an oblivious point on it.
\end{theorem}
\begin{proof}
Masur and Smillie \cite[Theorem 2(c)]{MS} state that every stratum of semitranslation surfaces except for $\Q(4), \Q(3, 1), \Q(1, -1)$ and $\Q(0)$ is non-empty.

\textbf{Case I: Even genus $\ge 4$}

In particular, for every $m \in\N$ the stratum $\Q(-1, 4m+1)$ is nonempty. For every $X\in \Q(-1, 4m+1)$, let $Y$ be the canonical double cover of $X$. Since $X$ has a single $\pi$ cone point, by Theorem \ref{NPS3.1}, $Y$ contains an oblivious point. By Kontsevich and Zorich \cite[Lemma 1]{KZ} $Y\in\H(4m+2)$. By combinatorial Gauss--Bonnet, $Y$ is a surface of genus $2m+2$.

\textbf{Case II: Odd genus $\ge 5$}
    
By the same theorem of Masur and Smillie, for every $ m\in \N$ the stratum $\Q(-1, 1, 1, 4n-1)$ is nonempty. For every  $X\in \Q(-1, 1, 1, 4m-1)$, let $Y$ be the canonical double cover of $X$. Since $X$ has a single $\pi$ cone point, by Theorem \ref{NPS3.1}, $Y$ contains an oblivious point. By Kontsevich and Zorich \cite[Lemma 1]{KZ}, $Y\in\H(2, 2, 4m)$. By combinatorial Gauss--Bonnet, $Y$ is a surface of genus $2m+3$.
\end{proof}

\begin{proposition}\label{4.4}
There are no surfaces with an oblivious point in genera 1 and 2. 
\end{proposition}
\begin{proof}
Kontsevich and Zorich \cite[Theorem 2]{KZ} state that every stratum in genus 1 and 2 is connected and coincides with its hyperelliptic component.Nygyen, Pan and Su \cite[Theorem 3]{NPS} state that no surface in the hyperelliptic component of any stratum admits an oblivious point. Hence no surface of genus 1 or 2 contains an oblivious point
\end{proof}

\begin{remark}\label{4.5}
No genus 3 surface containing an oblivious point can be constructed through Theorem \ref{NPS3.1}.
The argument in Theorem \ref{genus1} fails in genus 3. Even though the relevant stratum, $\Q(1, 1, -1, -1)$, is non-empty, this stratum fails the hypotheses of Theorem \ref{NPS3.1} as it admits two cone points of cone angle $\pi$.

Table~\ref{tabs} lists all possible strata of translation surfaces in genus 3, and the strata of semitranslation surfaces they can double cover. Since none of the semitranslation surface strata have exactly one cone point of cone angle $\pi$, no semitranslation surface stratum which can be double covered to produce a genus 3 surface satisfies the hypothesis of Theorem \ref{NPS3.1}.
\end{remark}

\begin{table}
\begin{center}
\begin{tabular}{ |c|c| } 
 \hline
A translation surface in & can be the canonical \\ 
& double cover of something in\\
 \hline
 $\H(4)$ & $\Q(3, -1^{3n + 4})$  \\ 
 $\H(3, 1)$& $\varnothing$  \\ 
 $\H(2, 2)$& $\Q(4, -1^{4n}), \Q(1, 1, -1^{4n+2})$\\
 $\H(2, 1, 1)$ & $\Q(1, 2, -1^{4n+3})$\\
 $\H(1, 1, 1, 1)$ & $\Q(2, 2, -1^4)$\\
 \hline
\end{tabular}
\caption{Possible strata for a genus 3 translation surface}
\label{tabs}
\end{center}
\end{table}

\section{Blocking Sets to Oblivious Points}\label{sec:block}

We provide an alternate construction that may be used to prove the main results from the previous section, generating oblivious points by branching over specific points on a square torus. This technique works in all genus $g \geq 3$, thereby proving Theorem~\ref{g geq 3}, which may not be proved by the previous methods as shown.
\subsection{Blocking Sets}

For any surface $X$ and any point $x\in X$, a finite set of points $P = \{p_1,\dots,p_n\}\subseteq X$ is called a \textit{blocking set} of $x$ if every closed geodesic through $x$ contains a point in the blocking set. $x$ is said to be blocked by $P$.

The set $P = \{(0, \frac12), (\frac12, 0), (\frac12, \frac12) \}$ is known to block the point $(0, 0)$ on the unit square torus (see Figure~\ref{blockset}). For a detailed discussion of blocking sets on tori, see Leli\`evre, Monteil, and Weiss \cite[Section 6]{LMW}. \\

\begin{figure}[h!]
\begin{center}
    \begin{tikzpicture}
    \draw[color=red] (0,0) -- (4,0);
    \draw[color=blue](4, 0) -- (4,4);
    \draw[color=red](4, 4) --(0, 4);
    \draw[color=blue](0, 4)--(0, 0);
    \draw (0, 0) node{O};
    \draw (2, 0) node{$p_1$};
    \draw (0, 2) node{$p_2$};
    \draw (2, 2) node{$p_3$};
    \end{tikzpicture}
    \caption{A blocking set for $O$}
    \label{blockset}
\end{center}
\end{figure}
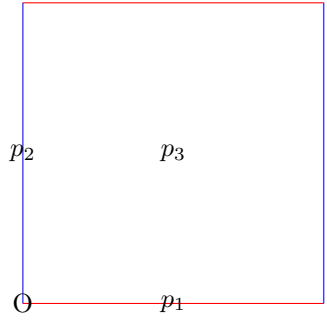

 \begin{lemma}\label{ob2}
Let $X$ be a translation surface with point $x$ and blocking set $P$ of $x$. Suppose $Y$ is a fully ramified branched translation cover with a translation cover $\rho: Y \rightarrow X$ with $\rho^{-1}(P)$ being cone points. Then any point in $\rho^{-1}(x)$ is oblivious.
\end{lemma} 
\begin{proof}
Any closed geodesic through $x$ will contain a point in $P$. Suppose $y \in \rho^{-1}(x)$ is not oblivious. Then there is some non-singular closed geodesic $\gamma$ containing $y$ in $Y$. By Lemma $3.5$ $\rho(\gamma)$ is a closed geodesic containing $x$. However, this implies that $\rho(\gamma) \cap P \neq \emptyset$. Thus $\gamma \cap \rho^{-1}(P) \neq \emptyset$. By assumption any point in $\rho^{-1}(P)$ is a cone point, but $\gamma$ is non-singular, a contradiction.
\end{proof} 

Using the slit construction, create a translation surface $Y$ with four cone points, each with cone angle $4\pi$ (see Figure~\ref{block}). This is achieved by identifying slits on two square tori with marked points at $P=\{(0,1/2), (1/4,1/2), (1/2,1/2), (1/2,0)\}$. As $P$ is a blocking set on each torus for the point at the origin, any closed geodesic on $Y$ containing $y_1$ or $y_2$ must pass through a cone point.

\begin{figure}[h!]
\begin{center}
    \begin{tikzpicture}
    \draw[color=red] (0,0) -- (4,0);
    \draw[color=blue](4, 0) -- (4,4);
    \draw[color=red](4, 4) --(0, 4);
    \draw[color=blue](0, 4)--(0, 0);
    \draw (0, 0) node{$y_1$};
    \draw (2, 0) node{X} -- (2, 2) node{X};
    \draw (0, 2) node{X} -- (1, 2) node{X};
    \draw[color=green] (5,0) -- (9,0);
    \draw[color=orange](9, 0) -- (9,4);
    \draw[color=green](9, 4) --(5, 4);
    \draw[color=orange](5, 4)--(5, 0);
    \draw (5, 0) node{$y_2$};
    \draw (7, 0) node{X} -- (7, 2) node{X};
    \draw (5, 2) node{X} -- (6, 2) node{X};
    \draw (0.5, 2.3) node{$A_1$};
    \draw (0.5, 1.8) node{$A_2$};
    \draw (1.7, 1) node{$B_1$};
    \draw (2.2, 1) node{$B_2$};
    \draw (5.5, 2.3) node{$A_2$};
    \draw (5.5, 1.8) node{$A_1$};
    \draw (6.7, 1) node{$B_2$};
    \draw (7.2, 1) node{$B_1$};
    
    \end{tikzpicture}
    \caption{Two tori connected via slits}
    \label{block}
\end{center}
\end{figure}
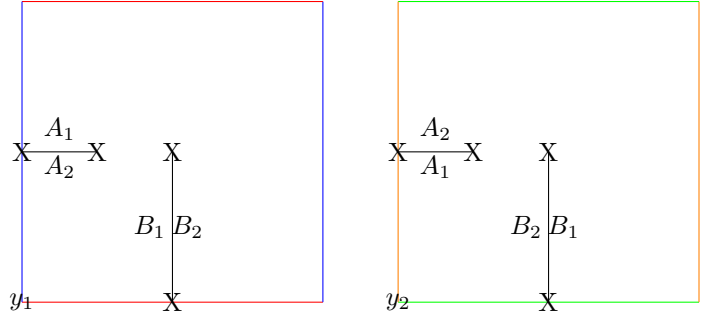

Now we generalise this construction to prove the main results from the previous section. We note in the lemma $\rho$ by definition is fully ramified and hence the following constructions will produce fully ramified covers.

\begin{construction}[2 oblivious points]
Consider the square torus $X$ with marked points $P = \{p_1 = (\frac12, 0), p_2 = (0, \frac12), p_3=(\frac12, \frac12), p_4=(\frac14, \frac12)\}$ and $x$ the point at the origin. Since $P$ contains a blocking set, $x$ is blocked by $P$.

Now consider $ (Y, \rho)$, a translation double cover of $X$ branched over $P$ with slit identification as shown in Figure 4. The points marked with ``X" are the preimages of the $p_i$ and are cone points of angle $4\pi$ each. Since $x$ is blocked in $(X, P)$ and $\rho^{-1}\{x\} = \{y_1, y_2\}$, Lemma~\ref{ob2} implies that $y_1$ and $y_2$ are both oblivious in $Y$.
\end{construction}

\begin{construction}[n oblivious points]
Consider $ (Z, \sigma: Z\to X)$, a degree $n$ translation cover of $X$ branching over $P$ with cyclic slit identification as shown in Figure~\ref{nslits}. 

\begin{figure}[h!]
\begin{center}
    \begin{tikzpicture}
    \draw[color=red] (0,0) -- (4,0);
    \draw[color=blue](4, 0) -- (4,4);
    \draw[color=red](4, 4) --(0, 4);
    \draw[color=blue](0, 4)--(0, 0);
    \draw (0, 0) node{$x_1$};
    \draw (2, 0) node{X} -- (2, 2) node{X};
    \draw (0, 2) node{X} -- (1, 2) node{X};
    \draw[color=green] (5,0) -- (9,0);
    \draw[color=orange](9, 0) -- (9,4);
    \draw[color=green](9, 4) --(5, 4);
    \draw[color=orange](5, 4)--(5, 0);
    \draw (5, 0) node{$x_2$};
    \draw (7, 0) node{X} -- (7, 2) node{X};
    \draw (5, 2) node{X} -- (6, 2) node{X};
    \draw (0.5, 2.3) node{$A_1$};
    \draw (0.5, 1.8) node{$A_2$};
    \draw (1.7, 1) node{$B_1$};
    \draw (2.2, 1) node{$B_2$};
    \draw (5.5, 2.3) node{$A_2$};
    \draw (5.5, 1.8) node{$A_3$};
    \draw (6.7, 1) node{$B_2$};
    \draw (7.2, 1) node{$B_3$};
\end{tikzpicture}
\end{center}

\begin{flushright}
    \begin{tikzpicture}
    \foreach \point in {(-1,0),(-0.85,0),(-0.7,0)}{
    \fill \point circle (2pt);
    }
    \draw[color=violet] (0,0) -- (4,0);
    \draw[color=olive](4, 0) -- (4,4);
    \draw[color=violet](4, 4) --(0, 4);
    \draw[color=olive](0, 4)--(0, 0);
    \draw (0, 0) node{$x_{n-1}$};
    \draw (2, 0) node{X} -- (2, 2) node{X};
    \draw (0, 2) node{X} -- (1, 2) node{X};
    \draw[color=magenta] (5,0) -- (9,0);
    \draw[color=teal](9, 0) -- (9,4);
    \draw[color= magenta](9, 4) --(5, 4);
    \draw[color=teal](5, 4)--(5, 0);
    \draw (5, 0) node{$x_n$};
    \draw (7, 0) node{X} -- (7, 2) node{X};
    \draw (5, 2) node{X} -- (6, 2) node{X};
    \draw (0.5, 2.3) node{$A_{n-1}$};
    \draw (0.5, 1.8) node{$A_n$};
    \draw (1.6, 1) node{$B_{n-1}$};
    \draw (2.2, 1) node{$B_n$};
    \draw (5.5, 2.3) node{$A_n$};
    \draw (5.5, 1.8) node{$A_1$};
    \draw (6.7, 1) node{$B_n$};
    \draw (7.2, 1) node{$B_1$};
    \end{tikzpicture}
\end{flushright}
\caption{A construction with $n$ oblivious points}
\label{nslits}
\end{figure}
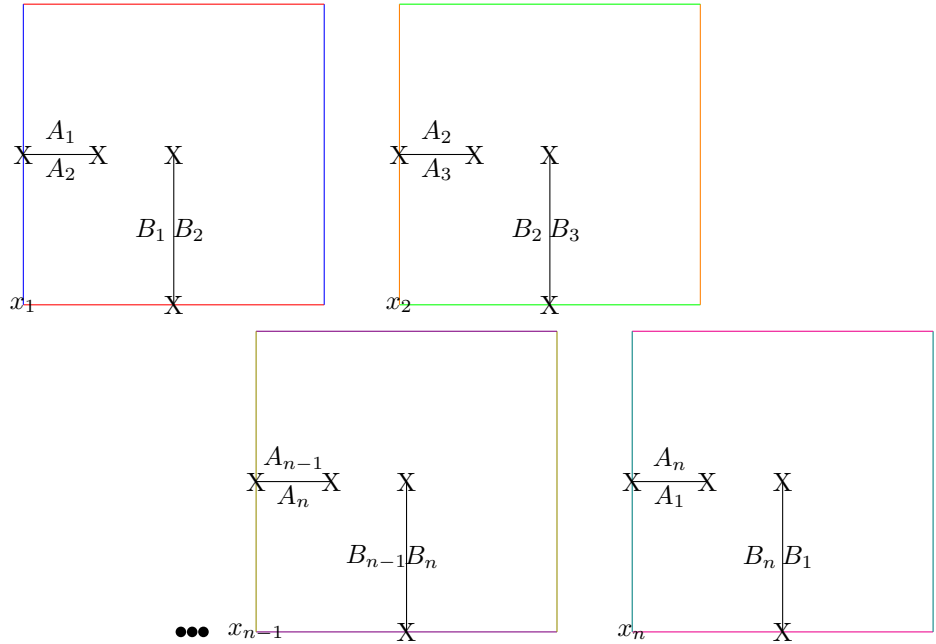

Note that there are only $4$ cone points each with angle $2\pi n$ at the identified preimages of the original $4$ marked points. We have $\rho^{-1}\{x\} = \{x_1,\dots, x_n\}$ and since $x$ is blocked in $(X, P)$, by Lemma \ref{ob2}, all the $x_i$ are oblivious in $Y$. By Lemma \ref{Cover Converse}, we know that all non-oblivious points in $X$ have non-oblivious preimages, so our cover construction has exactly $n$ oblivious points. 
\end{construction}

One is finally able to prove Theorem~\ref{g geq 3}.

\begin{proof}[Proof of Theorem~\ref{g geq 3}]
The previous construction yields $4$ cone points of angle $2\pi n$. For $n=2$ this amounts to having angle excess $8\pi$ which corresponds to a genus $3$ surface by Gauss-Bonnet. Together with our Theorem~\ref{genus1} this proves Theorem~\ref{g geq 3}.

In general, the construction yields a genus $2n-1$ surface by Gauss-Bonnet. To obtain even genera one may take our genus $2n-1$ construction and add a slit between any pair of squares from the coordinates $(5/8,7/8)$ to $(7/8,5/8)$. This slit adds 2 cone points with angle excess $2\pi$ each, therefore resulting in a genus $2n$ surface. 
\end{proof}

\section{Tilings and Oblivious Points}\label{sec:tilings}
\begin{corollary}
For every $n\ge 4$, there is a square tiled surface with $2n$ tiles and exactly one oblivious point.
\end{corollary}

\begin{figure}[h!]
\begin{center}
    \begin{tikzpicture}
    \draw (0, 0) rectangle (1,1);
    \draw (1, 0) rectangle (2, 1);
    \draw (0, 1) rectangle (1, 2);
    \draw (0, 2) rectangle (1, 3);
    \draw[-{Latex[length=2.3mm]}] (0, 2) -- (0, 2.5);
    \draw[-{Latex[length=2.3mm]}] (0, 2) -- (0, 1.5);
    \draw[-{Latex[length=2.3mm]}] (1, 2) -- (1, 1.6);
    \draw[-{Latex[length=2.3mm]}] (1, 2) -- (1, 1.45);
    \draw[-{Latex[length=2.3mm]}] (2, 0) -- (2, 0.45);
    \draw[-{Latex[length=2.3mm]}] (2, 0) -- (2, 0.6);
    \draw[color=red](0, 0) -- (0, 1);
    \draw[color=red](1, 2) -- (1, 3);
    \draw (2, 0) -- (2, .6);
    \end{tikzpicture}$\quad$
    \begin{tikzpicture}
    \draw [-{Latex[length=2.3mm]}] (0, 1.5) -- (1, 1.5);
    \draw (0, 0) -- (0, 0.0001);
    \end{tikzpicture}$\quad$
    \begin{tikzpicture}
      \draw (0, 0) rectangle (1,1);
    \draw (1, 0) rectangle (2, 1);
    \draw (0, 1) rectangle (1, 2);
    \draw (0, 2) rectangle (1, 3);
    \draw (2, 0) rectangle (3, 1);
    \draw[color=red](0, 0) -- (0, 1);
    \draw[color=red](1, 2) -- (1, 3);
 \draw[-{Latex[length=2.3mm]}] (0, 2) -- (0, 2.5);
    \draw[-{Latex[length=2.3mm]}] (0, 2) -- (0, 1.5);
    \draw[-{Latex[length=2.3mm]}] (1, 2) -- (1, 1.6);
    \draw[-{Latex[length=2.3mm]}] (1, 2) -- (1, 1.45);
    \draw[-{Latex[length=2.3mm]}] (3, 0) -- (3, 0.45);
    \draw[-{Latex[length=2.3mm]}] (3, 0) -- (3, .6);
    \draw (2, 0) node {A};
    \end{tikzpicture}
\end{center}
\caption{Semi-translation surfaces used to construct translation surfaces with $2n$ tiles and one oblivious point}
\end{figure}
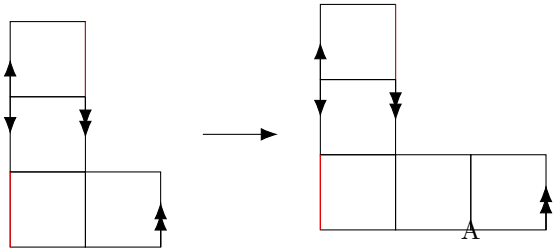
\begin{proof}
If one takes the canonical double cover of the 4-square L shown above, one gets an 8-square STS with one oblivious point by Theorem \ref{NPS3.1} and Lemma \ref{NPS3.2i} (the conditions of Lemma \ref{NPS3.2i} are met with Lemma \ref{polygon vertex} because all vertex points are one cone point). One may add a square to this base L to get a 5-tiled L which becomes a 10-square STS with an oblivious point (the conditions of Lemma \ref{NPS3.2i} are still met because only one vertex point, $A$, has been added and it is contained within the simple closed geodesic that is the vertical line shown). If one continues to add squares in this manner, one can get an even-tiled STS with an oblivious point for any even number greater than or equal to 8.
\end{proof}

\begin{lemma}\label{6.2}
There is no square-tiled surface with two or four tiles and an oblivious point
\end{lemma}
\begin{proof}
Proposition 19 of Matheus \cite{matheus} states that a square tiled surface in $\H(k_1, \dots , k_n)$ is tiled by at least $N = \displaystyle \sum_{i=1}^n(k_n + 1)$ squares.

Hence, every two tiled surface is in $\H(0)$, i.e. has genus 1. Every four tiled surface is in $\H(0), \H(1, 1)$ or $\H(2)$ i.e. has genus 1 or 2. By Proposition $\ref{4.4}$ there is no oblivious point on any surface of genus 1 or 2, hence no two or four tiled STS has an oblivious point. 
\end{proof}

\begin{remark}
    For the same reason as Lemma \ref{6.2} every six tiled surface is in  one of the following strata: $\H(0), \H(1, 1), \H(2), \H(2, 2), \H(1, 3), \H(4)$, all of which have genus $\le 3$. However, by Proposition \ref{4.4} no surface of genus 1 or 2 has an oblivious point on it, hence if such a surface exists, it must be in one of the following strata: $\H(4), \H(1, 3)$ or $\H(2, 2)$. No surface with an oblivious point in any of these strata (and consequently with six tiles) can be obtained through Proposition \ref{NPS3.1}, however, the existence of such a surface in general is unknown.
\end{remark}

\begin{corollary}
For every odd integer $m=kn^2, k\geq3, n\geq3,$ there exists a square tiled surface with $m$ tiles and exactly $k$ oblivious points on it.
\end{corollary}

\begin{proof}
To construct such a surface, one generalises the slit construction from Section 5. On the unit square torus one marks all the points at coordinates $\displaystyle \qty(\frac{p}{n}, \frac qn), p, q\in \qty{1, 2, \dots, n}$. By Leli\`evre, Monteil, and Weiss [5, Section 6], tbis set constitutes a blocking set for the origin. 

Hence, by taking a $k$--cover of this surface and slitting the $k$ copies together as was done in Section 5, one obtains a a translation surface with exactly $k$ oblivious points, one at the origin of each sheet. An example of the construction with $k\times 3^2$ tiles is shown in Figure \ref{kn^2}.

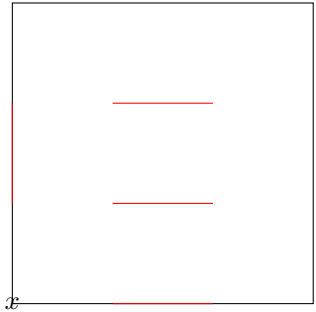
\begin{figure}[h!]
    \begin{center}
    \begin{tikzpicture}
    \draw (0, 0) rectangle (4,4);
    \draw[color=red] (4/3, 0) -- (8/3, 0);
    \draw[color=red] (4/3, 4/3) -- (8/3, 4/3);
    \draw[color=red] (4/3, 8/3) -- (8/3, 8/3);
    \draw[color=red] (0, 4/3) -- (0, 8/3);
    \draw (0, 0) node {$x$};
\end{tikzpicture}
\end{center}

    \caption{Slitting $k$ such sheets and identifying along a $k$--cycle gives a $k\times 3^2$ tiled surface and $k$ oblivious points, one at each $x$}
    \label{kn^2}
\end{figure}

This construction gives one the requisite surface with $kn^2$ tiles and $k$ oblivious points. This surface is in $\H((k-1)^{4(2n+1)})$ and has genus $2(2n+1)(k-1) + 1$
\end{proof}

Let us conclude with a final question: for large enough $n$, can one always construct a STS from $n$ squares which has an oblivious point?

\begin{question}
Is there an $N$ such that for all $n \ge N$, there exists a square-tiled surface with exactly $n$ tiles and an oblivious point?
\end{question}

\section*{Acknowledgements}
We would like to thank Aaron Calderon for his invaluable advice and countless helpful discussions related to this research, as well as his constructive feedback and support in editing this paper. We also extend our gratitude to Diana Davis for her conversations and suggestions on future research. We thank the Summer Undergraduate Math Research at Yale (SUMRY) program and the Young Mathematicians Conference (YMC) for their financial support of this work.  We also thank the referee for their considered suggestions.

\bibliographystyle{splncs04}
\bibliography{biblio}

\begin{thebibliography}{}
\bibitem{NPS} 
Nguyen, D., Pan, H., Su, W. (2017). \textit{Existence of closed geodesics through a regular point on translation surfaces}. Mathematische Annalen. 10.1007/s00208-019-01897-2. 

\bibitem{zorich}
Zorich, Anton. (2006) \textit{Flat Surfaces}. Frontiers in Number Theory, Physics, and Geometry Vol.I,
P. Cartier; B. Julia; P. Moussa; P. Vanhove (Editors),
Springer Verlag, 2006, 439–586.

\bibitem{MS}
Masur, H., Smillie, J. (1993) \textit{Quadratic differentials with prescribed singularities and pseudo-Anosov diffeomorphisms}. Commentarii Mathematici Helvetici 68, 289–307. https://doi.org/10.1007/BF02565820

\bibitem{KZ}
Kontsevich, M., Zorich, A. (2003) \textit{Connected components of the moduli spaces of Abelian differentials with prescribed singularities.} Invent. math. 153, 631–678. https://doi.org/10.1007/s00222-003-0303-x

\bibitem{LMW}
Lelièvre, S., Monteil, T., Weiss, B. (2016) \textit{Everything is illuminated}. Geom. Topol. 20, no. 3, 1737--1762. doi:10.2140/gt.2016.20.1737. https://projecteuclid.org/euclid.gt/1510859001

\bibitem{matheus}
Matheus, Carlos. (2018) \textit{Three Lectures on Square Tiled} Surfaces. Institut Fourier, Grenoble, France. “Teichmu\"uller dynamics, mapping class groups and applications”

\end{thebibliography}

\end{document}